\documentclass[reqno,a4paper, 11pt]{amsart}

\usepackage[a4paper=true,pdfpagelabels]{hyperref}
\usepackage{graphicx}

\usepackage[ansinew]{inputenc}
\usepackage{mathabx}

\usepackage[T1]{fontenc}

\usepackage{lmodern}
\usepackage{amsfonts,epsfig}
\usepackage{latexsym}
\usepackage{amsmath}
\usepackage{amssymb}
\usepackage{color}

\newtheorem{theorem}{Theorem}
\newtheorem{lemma}[theorem]{Lemma}

\newtheorem{question}[theorem]{Question}

\newtheorem{lettertheorem}{Theorem}
\newtheorem{letterlemma}[lettertheorem]{Lemma}

\theoremstyle{definition}

\theoremstyle{remark}

\numberwithin{equation}{section}

\newcommand{\D}{\mathbb{D}}
\newcommand{\DD}{\widehat{\mathcal{D}}}
\newcommand{\Dd}{\widecheck{\mathcal{D}}}
\newcommand{\DDD}{\mathcal{D}}
\newcommand{\N}{\mathbb{N}}
\newcommand{\Z}{\mathbb{Z}}

\newcommand{\C}{\mathbb{C}}

\renewcommand{\phi}{\varphi}

\newcommand{\rad}{{\rm rad}}

\def\a{\alpha}               
           
     \def\om{\omega}      
       \def\t{\theta}       
         \def\r{\rho}         \def\z{\zeta}

\def\omgu{\widehat{\omega_1}}
\def\nugu{\widehat{\nu_1}}
\def\sgu{\widehat{\sigma_1}}

\renewcommand{\H}{\mathcal{H}}

\newenvironment{Prf}{\noindent{\emph{Proof of}}}
{\hfill$\Box$ }
\begin{document}

\title[One weight inequality for Bergman projection and Calder\'on operator]{One weight inequality for Bergman projection and Calder\'on operator induced by radial weight}

\date{\today}

\begin{abstract}
Let $\omega$ and $\nu$ be radial weights on the unit disc of the complex plane such that $\omega$ admits the doubling property $\sup_{0\le r<1}\frac{\int_r^1 \omega(s)\,ds}{\int_{\frac{1+r}{2}}^1 \omega(s)\,ds}<\infty$. Consider the one weight inequality
	\begin{equation}\label{ab1}
  \|P_\omega(f)\|_{L^p_\nu}\le C\|f\|_{L^p_\nu},\quad 1<p<\infty,\tag{\dag}
  \end{equation}
for the Bergman projection $P_\omega$ induced by $\om$. It is shown that the Muckenhoupt-type condition 
	$$
	A_p(\omega,\nu)=\sup_{0\le r<1}\frac{\left(\int_r^1 s\nu(s)\,ds \right)^{\frac{1}{p}}\left(\int_r^1 s\left(\frac{\omega(s)}{\nu(s)^{\frac1p}}\right)^{p'}\,ds \right)^{\frac{1}{p'}}}{\int_r^1 s\omega(s)\,ds}<\infty,
	$$
is necessary for \eqref{ab1} to hold, and sufficient if $\nu$ is of the form $\nu(s)=\omega(s)\left(\int_r^1 s\omega(s)\,ds \right)^\alpha$ for some $-1<\alpha<\infty$. This result extends the classical theorem due to Forelli and Rudin for a much larger class of weights. In addition, it is shown that for any pair $(\omega,\nu)$ of radial weights the Calder\'on operator 
	$$ 
	H^\star_\omega(f)(z)+H_\omega(f)(z)
	=\int_{0}^{|z|} f\left(s\frac{z}{|z|}\right)\frac{s\omega(s)\,ds}{\int_s^1 t\omega(t)\,dt}
	+\frac{\int_{|z|}^1f\left(s\frac{z}{|z|}\right) s\omega(s)\,ds}{\int_{|z|}^1 s\omega(s)\,ds}\,ds
	$$ 
is bounded on $L^p_\nu$ if and only if $A_p(\omega,\nu)<\infty$.
\end{abstract}


\author{Francisco J. Mart\'{\i}n Reyes}
\address{Departamento de An\'alisis Matem\'atico, Universidad de M\'alaga, Campus de
Teatinos, 29071 M\'alaga, Spain}
 \email{martin\_reyes@uma.es}

\author{Pedro Ortega}
\address{Departamento de An\'alisis Matem\'atico, Universidad de M\'alaga, Campus de
Teatinos, 29071 M\'alaga, Spain} \email{portega@uma.es}

\author{Jos\'e \'Angel Pel\'aez}
\address{Departamento de An\'alisis Matem\'atico, Universidad de M\'alaga, Campus de
Teatinos, 29071 M\'alaga, Spain} \email{japelaez@uma.es}

\author{Jouni R\"atty\"a}
\address{University of Eastern Finland, P.O.Box 111, 80101 Joensuu, Finland}
\email{jouni.rattya@uef.fi}

\thanks{The first three authors are supported in part by Ministerio de Econom\'{\i}a y Competitividad, Spain, projects
PGC2018-096166-B-100; Junta de Andaluc{\'i}a, projects FQM210, FQM354 and UMA18-FEDERJA-002.}


\maketitle

\section{Introduction and main results}

Let $\H(\D)$ denote the space of analytic functions in the unit disc $\D=\{z\in\C:|z|<1\}$. For a nonnegative function $\om\in L^1([0,1))$, its extension to $\D$, defined by 
$\om(z)=\om(|z|)$ for all $z\in\D$, is called a radial weight. For $0<p<\infty$ and such an $\omega$, the Lebesgue space $L^p_\om$ consists of complex-valued measurable functions $f$ on $\D$ such that
    $$
    \|f\|_{L^p_\omega}^p=\int_\D|f(z)|^p\omega(z)\,dA(z)<\infty,
    $$
where $dA(z)=\frac{dx\,dy}{\pi}$ is the normalized Lebesgue area measure on $\D$. The corresponding weighted Bergman space is $A^p_\om=L^p_\omega\cap\H(\D)$. Throughout this paper we assume $\widehat{\om}(z)=\int_{|z|}^1\om(s)\,ds>0$ for all $z\in\D$, for otherwise $A^p_\om=\H(\D)$. As usual, we write~$A^p_\alpha$ for the classical weighted Bergman space induced by the standard radial weight $(\alpha+1)(1-|z|^2)^\alpha$, where $-1<\alpha<\infty$.

The Bergman Hilbert space $A^2_\om$ is a closed subspace of $L^2_\om$, and the orthogonal Bergman projection from $L^2_\om$ to $A^2_\om$ is the integral operator
  \begin{equation*}\label{intoper}
    P_\om(f)(z)=\int_{\D}f(\z)\overline{B^\om_{z}(\z)}\,\om(\z)dA(\z),\quad z\in\D,
    \end{equation*}
where $\{B^\om_{z}\}_{z\in\D}$ are the reproducing kernels of $A^2_\om$. The kernel of the classical weighted Bergman space $A^2_\alpha$ is denoted by $B^\alpha_z$, and $P_\alpha$ stands for the corresponding Bergman projection. 

This paper concerns the question of when, for a given $1<p<\infty$, we have 
		\begin{equation}\label{twoweight}
    \|P_\om(f)\|_{L^p_\nu}\le C\|f\|_{L^p_\nu},\quad f\in L^p_\nu.
    \end{equation}
We also consider the one-weight inequality analogous to \eqref{twoweight} for the maximal Bergman projection
	\begin{equation*}
	P^+_\om(f)(z)=\int_{\D}f(\z)|B^\om_{z}(\z)|\om(\z)\,dA(\z),\quad z\in\D.
  \end{equation*}
The boundedness of projections on $L^p$-spaces is a compelling topic which has attracted a lot of attention during the last decades. Regarding the inequality \eqref{twoweight}, we refer to \cite{AlPoRe,B1981,BB,PelRatKernels,PRW,PelaezRattya2019}. The most known result concerning \eqref{twoweight} is due to Bekoll\'e and Bonami~\cite{B1981,BB}, and it concerns the case when~$\nu$ is an arbitrary weight and the inducing weight $\om$ is standard. See \cite{ACJFA12,PRW,PottRegueraJFA13} for recent extensions of this result. Some of the key tools in 
the proof of the aforementioned result of Bekoll\'e and Bonami are based on the classical harmonic analysis, and rely strongly on the nice properties that the standard kernel has. In particular, the neat explicit formula $B^\alpha_z(\z)=(1-\overline{z}\z)^{-(2+\alpha)}$ shows that each $B^\alpha_z$ is zero-free and its modulus is essentially constant in hyperbolically bounded regions.
However, these properties are no longer true for the Bergman reproducing kernel induced by an arbitrary radial weight. Namely, for a given radial weight $\om$ there exists another radial weight $\nu$ such that $A^2_\om=A^2_\nu$, but $B^\nu_z$ have zeros, see the proof of \text\\ 
\cite[Theorem~2]{BoudreauxJGA19}. Going further, one of the main obstacles to tackle the inequality~\eqref{twoweight} is the lack of explicit expressions for $B^\om_z$. By and large, this kernel has the representation $B^\om_z(\z)=\sum \overline{e_n(z)}e_n(\z)$ for each orthonormal basis $\{e_n\}$ of $A^2_\om$. Therefore the normalized monomials show that the properties of the weight $\om$ are transmitted to the projection $P_\om$ through the representation $B^\om_z(\z)=\sum_{n=0}^\infty\frac{\left(\overline{z}\z\right)^n}{2\om_{2n+1}}$ of the kernel via its odd moments $\om_{2n+1}$. In general, from now on we write $\om_x=\int_0^1r^x\om(r)\,dr$ for all $x\ge0$. 
	 
To the best of our knowledge, describing the radial weights $\om$ such that $P_\om$ is bounded on~$L^p_\om$ is an open problem that was formally posed by Dostani\'c in \cite{Dostanic} almost two decades ago. Even though this particular case $\om=\nu$ of \eqref{twoweight} remains unsettled, it was recently proved in \cite[Theorem~7]{PelaezRattya2019} that $P_\om:\,L^p_\om\to L^p_\om$, $1<p<\infty$, is bounded if $\om$ belongs to $\widehat{\mathcal{D}}$. The class $\widehat{\mathcal{D}}$ consists of the radial weights $\om$ having the doubling property $\sup_{0\le r<1}\frac{\widehat{\omega}(r)}{\widehat{\om}\left(\frac{1+r}{2}\right)}<\infty$. Further, mild extra regularity hypothesis on $\om$ guarantee the necessity of the doubling condition $\om\in\DD$ for $P_\om:\,L^p_\om\to L^p_\om$ to be bounded~\cite[Corollary~8]{PelaezRattya2019}. Due to these results we will study the inequality \eqref{twoweight} under the initial hypothesis $\om\in\DD$. Another class of doubling weights that naturally appears in our study is denoted by $\Dd$, and it consists of radial weights $\om$ for which there exist constants $K=K(\om)>1$ and $C=C(\om)>1$ such that $\widehat{\om}(r)\ge C\widehat{\om}\left(1-\frac{1-r}{K}\right)$ for all $0\le r<1$. The intersection $\DD\cap\Dd$ of these classes of weights is denoted by $\mathcal{D}$. The next few lines are dedicated to offer a brief insight to these classes. Each standard radial weight $(\alpha+1)(1-|z|^2)^\alpha$ belongs $\DDD$, while $\Dd\setminus\DDD$ contains weights that tend to zero exponentially. The class of rapidly increasing weights, introduced in \cite{PelRat}, lies entirely within $\DD\setminus\DDD$, and a typical example of such a weight is $\nu(z)=(1-|z|^2)^{-1}\left(\log\frac{e}{1-|z|^2}\right)^{-\alpha}$, where $\alpha>1$. To this end we emphasize that the containment in $\DD$ or $\Dd$ does not require continuity neither positivity. In fact, weights in these classes may vanish on a relatively large part of each outer annulus $\{z:r\le|z|<1\}$ of $\D$. For basic properties of aforementioned classes, concrete nontrivial examples and more, see \cite{PelSum14,PelRat,PelaezRattya2019} and the relevant references therein.

The recent papers \cite{KPR2,PGR} show that the (radial) averaging operator  
		$$
		H_\om(f)(z)=\frac{\int_{|z|}^1f\left(s\frac{z}{|z|}\right)\om_1(s)\,ds}{\widehat{\om_1}(|z|)}
		,\quad z\in\D\setminus\{0\},\quad \om_1(s)=s\om(s),  
    $$
can be considered as a toy model to study  the pairs of  radial weights $(\om,\nu)$  such that \eqref{twoweight} holds. In order to make this statement precise, for $1<p<\infty$ and radial weights $\om$ and $\nu$, write $\sigma(r)=\sigma_{p,\om,\nu}(r)=\left(\frac{\om(r)}{\nu(r)^{\frac1p}}\right)^{p'}$, and define
\begin{equation}\label{eq:Mpcondition}
    M_p(\om,\nu)
		=\sup_{0<r<1}\left(\int_0^r\frac{\nu_1(s)}{\omgu(s)^p}\,ds+1\right)^\frac1p\widehat{\sigma_1}(r)^{\frac{1}{p'}}.
    \end{equation}
With this notation the classical result of Muckenhoupt \cite[Theorem~2]{Muckenhoupt1972} states that $H_\om:L^p_\nu\to L^p_\nu$ is bounded if and only if $M_p(\om,\nu)<\infty$. This together with minor modifications
in the proof of \cite[Theorem~13]{PelaezRattya2019} yields the following result, which is a substantial improvement of \cite[Corollary~1.2]{KPR2}. 

\begin{lettertheorem}\label{th:A}
\label{Theorem:P_w-L^p_v}
Let $1<p<\infty$ and $\om\in\DD$, and let $\nu$ be a radial weight. Then the following statements are equivalent:
\begin{enumerate}
    \item[\rm(i)] $P^+_\om:L^p_\nu\to L^p_\nu$ is bounded;
    \item[\rm(ii)] $P_\om:L^p_\nu\to L^p_\nu$ is bounded and $\nu\in\Dd$;
    \item[\rm(iii)] $H_\om:L^p_\nu\to L^p_\nu$ is bounded and  $\om,\nu\in\DDD$;
    \item[\rm(iv)] $M_p(\om,\nu)<\infty$ and $\om,\nu\in\DDD$.
\end{enumerate}
\end{lettertheorem}

This result characterizes the one-weight inequality \eqref{twoweight} whenever  $\om\in\DD$ and $\nu\in\Dd$, and establishes clear connections between the projections $P_\om$ and $P^+_\om$, and the averaging operator~$H_\om$. What remains unsettled is \eqref{twoweight} when $\om\in\DD$ but the hypothesis $\nu\in\Dd$ is removed. In order to approach this question, we write 
		\begin{equation}\label{eq:apcondition}
    A_p(\om,\nu)=\sup_{0\le r<1}\frac{\nugu(r)^{\frac{1}{p}}\sgu(r)^{\frac{1}{p'}}}{\omgu(r)}.
    \end{equation}
It is known by \cite[Theorem~13]{PelaezRattya2019} that also the condition $A_p(\om,\nu)<\infty$ characterizes the bounded Bergman projections $P_\om:L^p_\nu\to L^p_\nu$, provided $\om\in\DD$ and $\nu\in\Dd$. Further, this condition implies $M_p(\om,\nu)<\infty$ for any pair of radial weights $(\om,\nu)$, see the proof of \cite[Theorem~13]{PelaezRattya2019} or Theorem~\ref{th:calderon} below. Going further, we will prove the following result.

\begin{letterlemma}\label{le:nec}
Let $1<p<\infty$ and $\om\in\DD$, and let $\nu$ be a radial weight. If $P_\om: L^p_\nu\to L^p_\nu$ is bounded, then $A_p(\om,\nu)<\infty$.
\end{letterlemma}

This lemma is essentially contained in the proof of \cite[Theorem~13]{PelaezRattya2019}, but we provide the details for the convenience of the reader and the sake of completeness. 

To get the big picture right, let us now consider the radial weights $\omega\equiv1$ and $\nu(s)=(1-s^2)^{-1}\left( \log\frac{e}{1-s^2}\right)^{-\alpha}$, where $\alpha>1$. Obviously, $\om\in\DD$, and straightforward calculations show that $\nu\notin\Dd$, $M_p(\om,\nu)<\infty$ but $A_p(\om,\nu)=\infty$. Therefore $P_\om$ is not bounded on $L^p_\nu$ for any $1<p<\infty$ by Lemma~\ref{le:nec}. This forces us to focus on the condition $A_p(\om,\nu)<\infty$ and rules out $M_p(\om,\nu)<\infty$ as a candidate to describe the couples of radial weights $(\om,\nu)$ such that \eqref{twoweight} holds under the initial hypothesis $\om\in\DD$. Unfortunately, we can not judge whether or not the reverse statement to Lemma~\ref{le:nec} is true. This is, in our opinion, an intriguing matter and deserves to be stated as an open question.

\begin{question}\label{q2}
Let $1<p<\infty$ and $\om\in\DD$, and let $\nu$ be a radial weight such that $A_p(\om,\nu)<\infty$. Is then $P_\om$ bounded on $L^p_\nu$?
\end{question}

Despite the fact that we are not able to answer Question~\ref{q2} in its full generality, we provide an affirmative answer in an important particular case.
 
\begin{theorem}\label{pr:potencias}
Let $1<p<\infty$, $-1<\alpha<\infty$, $\om\in\DD$ and $\nu=\om\widehat{\om_1}^\alpha$. Then $P_\om:L^p_\nu\to L^p_\nu$ is bounded if and only if $A_p(\om,\nu)<\infty$. Moreover, $P^+_\om:L^p_\nu\to L^p_\nu$ is bounded if and only if $A_p(\om,\nu)<\infty$ and $\om\in\DDD$. 
\end{theorem}

The classical result \cite[Theorem~4.24]{Zhu} for the standard weights, attributed to Forelli and Rudin~\cite{ForRud74}, states that $P_\gamma:L^p_\beta\to L^p_\beta$ is bounded if and only $p(\gamma+1)>\beta+1$. This result is just a particular case of Theorem~\ref{pr:potencias}. Namely, if $\om(s)=(1-s^2)^\gamma$ and $\nu(s)=(1-s^2)^\beta$, where $\gamma,\beta>-1$, then 
$\nu=(2(\gamma+1))^{-\alpha}\om\widehat{\om_1}^\alpha$ for $\alpha=\frac{\beta-\gamma}{\gamma+1}>-1$, and $p(\gamma+1)>\beta+1$ if and only if $A_p(\om,\nu)<\infty$. Since $\DD$ is a very rich class of radial weights, the first statement in Theorem~\ref{pr:potencias} is a far reaching extension of \cite[Theorem~4.24]{Zhu}. The second statement is just a consequence of \cite[Proposition~20]{PelaezRattya2019} and the first statement. It is well-known among experts working on operator theory on spaces of analytic functions that the aforementioned result due to Forelli and Rudin is an extremely useful tool with a plethora of applications within this area of research~\cite{Zhu}. Therefore we expect that Theorem~\ref{pr:potencias} boosts the development of the operator theory in Bergman spaces $A^p_\om$ induced by $\om\in\DD$.

It is very important to realize the distinction between $\DDD$ and $\DD\setminus\DDD$ when proving Theorem~\ref{pr:potencias}. Namely, since $P_\om$ and $P^+_\om$ are not always simultaneously bounded when $\DD\setminus\DDD$, 
any technique, where $|P_\om(f)|$ is dominated by $P^+_\om(|f|)$, is bound to fail all together to prove Theorem~\ref{pr:potencias}. In particular, this shows that arguments employed in the proofs of \cite[Theorem~11]{PelRatKernels} and \cite[Theorem~13]{PelaezRattya2019} do not serve us here. The techniques we will use to prove Theorem~\ref{pr:potencias} are by no means standard. Our first step towards the theorem is the following reformulation of \eqref{twoweight} in our setting via a duality relation between weighted Bergman spaces.
 
\begin{theorem}\label{th:Pwdospesos}
Let $1<p<\infty$ and $\om\in\DD$, and let $\nu$ be a radial weight. Then $P_\om:L^p_\nu\to L^p_\nu$ is bounded if and only if $\left(A^{p'}_\sigma\right)^\star\simeq A^p_\nu$ via the $A^2_\om$-pairing with equivalence of norms.
\end{theorem}

The next step in the proof of Theorem~\ref{pr:potencias} is to observe that the dual of~$A^{p'}_\sigma$ can be identified with the space of coefficient multipliers from $A^{p'}_\sigma$ to $H^\infty$ under the $H^2$-pairing (the Cauchy pairing). As usual, for $0<p\le\infty$, we write $H^p$ for the classical Hardy space. This observation becomes practical once we endow $A^p_\om$ with a suitable equivalent norm. The one that works here is an $\ell^p$-type norm of the Hardy norms of blocks of the Maclaurin series whose size depend on the weight~$\om$, see~\cite[Theorem~3.4]{PRAntequera} and also \cite[Theorem~4]{PelRathg}. These arguments allow us to show that there exists a space $X_p(\om,\nu)$ of analytic functions such that $\left(A^{p'}_\sigma\right)^\star\simeq X_p(\om,\nu)$ via the $H^2$-pairing. Finally, to get the desired pairing and space, we define the operator $I^\om(g)(z)=\sum_{k=0}^\infty\widehat{g}(k)\om_{2k+1}z^{k}$, and we prove that $I^\om: A^p_\nu\to X_p(\om,\nu)$ is isomorphic and onto. This establishes Theorem~\ref{pr:potencias}.

We now turn back to study relationships between the boundedness of the radial averaging operator $H_\om$ on $L^p_\nu$, and the condition $A_p(\om,\nu)<\infty$. Recall that $M_p(\om,\nu)<\infty$ describes the boundedness of $H_\om$ on $L^p_\nu$ by \cite[Theorem~2]{Muckenhoupt1972}, and $A_p(\om,\nu)<\infty$ always implies $M_p(\om,\nu)<\infty$. We proceed with necessary definitions.

The adjoint of $H_\om$, with respect to $\langle\cdot,\cdot\rangle_{L^2_\om}$, is the operator
	$$
	H_\om^\star(f)(z)=\int_{0}^{|z|} f\left(s\frac{z}{|z|}\right)\frac{\om_1(s)}{\omgu(s)}\,ds,\quad z\in\D\setminus\{0\}.
	$$  
For a radial weight $\eta$, the associated H\"ormander maximal function \cite{HormanderL67} and Stieltjes transfrom
\cite{Widder} are defined by
	\begin{equation}\label{eq:maximal}
	M_{\eta,\rad}(f)(z)
	=\sup_{b<|z|}\frac{\int_b^1 \left|f\left(s\frac{z}{|z|}\right)\right|\eta(s)\,ds}{\widehat{\eta}(b)},\quad z\in\D\setminus\{0\},
	\end{equation}
and 
	\begin{equation}\label{eq:Sti}
	S_{\eta,\rad}(f)(z)
	=\int_0^1\frac{f\left(s\frac{z}{|z|}\right) \eta(s)}{\widehat{\eta}(s)+\widehat{\eta}(|z|)}\,ds,\quad z\in\D\setminus\{0\},
	\end{equation}
	respectively.
Our findings concerning the averaging radial operators and the condition \eqref{eq:apcondition} are gathered in the next result. 	

\begin{theorem}\label{th:calderon}
Let $1<p<\infty$ and let $\om$ and $\nu$ be radial weights. Then the following statements are equivalent:
	\begin{enumerate}
	\item[\rm(i)] $H_\om+H_\om^\star :L^p_\nu\to L^p_\nu$ is bounded;
	\item[\rm(ii)] $S_{\om_1,\rad} :L^p_\nu\to L^p_\nu$ is bounded;
  \item[\rm(iii)] $M_{\om_1,\rad} :L^p_\nu\to L^p_\nu$ is bounded;
	\item[\rm(iv)] $A_p(\om,\nu)<\infty$.
	\end{enumerate}
Moreover,
	\begin{equation*}
	\|H_\om+H_\om^\star\|_{L^p_{\nu}\to L^p_{\nu}}
	\asymp\|S_{\om_1,\rad}\|_{L^p_{\nu}\to L^p_{\nu}}
	\asymp\|M_{\om_1,\rad}\|_{L^p_{\nu}\to L^p_{\nu}}
	\asymp A_p(\om,\nu).
	\end{equation*}
\end{theorem}   
As for the proof of Theorem~\ref{th:calderon}, the inequality $\|M_{\om_1,\rad}\|_{L^p_{\nu}\to L^p_{\nu}}\lesssim A_p(\om,\nu)$ boils down to a decomposition of the interval $[0,1)$ in level sets for the maximal function $M_{\om_1,\rad}f_\theta$, where $f_\theta(r)=f(re^{i\theta})$, together with some properties of the maximal function. Since (iii)$\Rightarrow$(iv) is deduced by standard techniques, we have the equivalence between (iii) and (iv). After that, we observe the pointwise inequalities 
	$$
	M_{\om_1,\rad}(f)(z)\le 2  S_{\om_1,\rad}(f)(z)\le H_\om(f)(z)+H_\om^\star(f)(z),\quad z\in\D,\quad f\ge 0,
	$$
which yield (i)$\Rightarrow$(ii)$\Rightarrow$(iii). Finally, the implication (iv)$\Rightarrow$(i) follows from the inequalities
$\|H_\om\|_{L^p_{\nu}}\le\|M_{\om_1,\rad}\|_{L^p_{\nu}\to L^p_{\nu}}\lesssim A_p(\om,\nu)$ and duality relations.

The remaining part of the paper is organized as follows. Section~\ref{s:analitica} contains the proofs of Lemma~\ref{le:nec}, Theorem~\ref{th:Pwdospesos} and Theorem~\ref{pr:potencias} in the said order. Theorem~\ref{th:calderon} is proved in Section~\ref{s:real}.

Throughout the paper we repeatedly use the fact that $\om\in\DD$ (resp. $\om\in\Dd$) if and only if $\om_1\in\DD$ (resp. $\om_1\in\Dd$). Further, as usual, for two non-negative functions $A$ and $B$, the notation $A\lesssim B$ (or $B\gtrsim A$) means that there exists a constant $C>0$, independent of the variables involved, such that $A\le C\,B$. Furthermore, we will write $A\asymp B$ when $A\lesssim B\lesssim A$. 

\section{Bergman projection}\label{s:analitica}

We begin with a detailed proof of the necessary condition \eqref{eq:apcondition} for $P_\om:L^p_\nu\to L^p_\nu$ to be bounded.

\medskip

\begin{Prf}{\em{Lemma~\ref{le:nec}.}} Assume that $P_\om:L^p_\nu\to L^p_\nu$ is bounded. Then $P_\omega(f)$ is well defined almost everywhere on $\D$ for each $f\in L^p_\nu$. It follows that 
	\begin{equation}\label{eq:i1}
	\int_{\D}|f(\z)||B^\om_z(\z)|\om(\z)\,dA(\z)<\infty,\quad f\in L^p_\nu, 
	\end{equation}
for almost every $z\in\D$. Now that $B^\om_{z}(\z)=\sum_{n=0}^\infty \frac{(\overline{z}\z)^n}{2\om_{2n+1}}$, with $\om_{2n+1}=\int_0^1 s^{2n+1}\om(s)\,ds$, we have $B^\om_0(\z)=\frac{1}{2\om_1}$. Therefore there exists $\delta=\delta(\om)>0$ such that $|B^\om_z(\z)|\ge\frac{1}{4\om_1}$ for all $\z\in\D$ and $z$ such that $|z|<\delta$. This together with \eqref{eq:i1} yields $L^p_\nu\subset L^1_\om$.
Let us show that the embedding is continuous. If this were not the case, then for each $n\in\N$ there would exist $f_n\in L^p_\nu$ such that $\|f_n\|_{L^p_\nu}=1$ and $\|f_n\|_{L^1_\om}\ge n^3$. Consider the function $F=\sum_{n=1}^\infty \frac{|f_n|}{n^2}$. It satisfies
	\begin{equation}\label{em1}
	\|F\|_{L^p_\nu}\le\sum_{n=1}^\infty \frac{\| f_n\|_{L^p_\nu}}{n^2}=\sum_{n=1}^\infty \frac{1}{n^2}<\infty
	\end{equation}
and $F\ge\frac{|f_n|}{n^2}$ for each $n\in\N$. Therefore $\|F\|_{L^1_\om}\ge\frac{\|f_n\|_{L^1_\om}}{n^2}\ge n$ for each $n\in\N$, and thus $F\notin L^1_\omega$. This together with \eqref{em1} contradicts $L^p_\nu\subset L^1_\om$. Consequently, there exists a constant $C=C(\om,\nu,p)>0$ such that
	\begin{equation}\label{eq:embedding}
	\|f\|_{L^1_\om}\le C\|f\|_{L^p_\nu},\quad f\in L^p_\nu.	
	\end{equation}
This inequality implies that $\om$ is absolutely continuous with respect to $\nu$, and hence the function $f_n=\min\{n,\left(\frac{\om}{\nu}\right)^{\frac{1}{p-1}}\}$ is well-defined almost everywhere on $\D$ for all $n\in\N$. Moreover, $f_n^p\nu\le f_n\om$ almost everywhere, and hence \eqref{eq:embedding} yields $\sup_n\|f_n\|_{L^1_\om}\le C^{p'}$. The monotone convergence theorem now shows that
  \begin{equation}\label{eq:ini}
  2\int_{0}^1 s\left(\frac{\om(s)}{\nu(s)^{\frac1p}}\right)^{p'}\,ds
	=\int_\D\lim_{n\to\infty}|f_n(z)|\om(z)\,dA(z)
	=\lim_{n\to\infty}\|f_n\|_{L^1_\om}<\infty.
\end{equation}

Recall that the adjoint operator $P^*_\om$ is defined via the identity 
	\begin{equation}\label{eq:adjointPomega}
	\langle P_\om(f), g \rangle_{L^2_\nu}= \langle f, P^*_\om(g) \rangle_{L^2_\nu}, \quad f\in L^p_\nu,\quad g\in  L^{p'}_\nu.
	\end{equation}
We next establish an integral formula for the image of the monomial $m_n(z)=z^n$ under $P^*_\om$. First observe that for each fixed $z\in\D$ the series $\sum_{k=0}^\infty\frac{(z\overline{\z})^k}{2\om_{2k+1}}$ converges uniformly on $\overline{\D}$. By using this and Fubini's theorem we deduce that each bounded function $f$ satisfies 
	\begin{equation}
	\begin{split}\label{eq:i2}
	\langle P_\om(f), m_n \rangle_{L^2_\nu}
	&=\int_{\D}\left(\int_\D f(\z)B^\om_\z(z)\om(\z)\,dA(\z) \right)\overline{m_n(z)}\nu(z)\,dA(z)\\
	&=\int_{\D}\left(\sum_{k=0}^\infty\frac{\left(\int_\D f(\z)\overline{\z}^k\om(\z)\,dA(\z) \right)z^k}{2\om_{2k+1}}\right)
	\overline{m_n(z)}\nu(z)\,dA(z)\\
	&=\int_{\D}f(\z)\frac{\left(\int_\D\left(z\overline{\z}\right)^n \overline{m_n(z)}\nu(z)\,dA(z)  \right)}{2\om_{2n+1}}\om(\z)\,dA(\z)\\
	&=\int_{\D}f(\z)\overline{\left( \frac{\om(\z)}{\nu(\z)}\int_\D m_n(z)B^\om_z(\z)\nu(z)\,dA(z)\right)}\nu(\z)\,dA(\z).
	\end{split}
	\end{equation}
Define 
	$$
	L_\om(g)(\z)=\frac{\om(\z)}{\nu(\z)}\int_\D g(z)B^\om_z(\z)\nu(z)\,dA(z),\quad\zeta\in\D,\quad g\in L^1_\nu.
	$$
Observe that $L_\om(m_n)(\z)=\frac{\om(\z)\nu_{2n+1}}{\nu(\z)\om_{2n+1}}\z^n$ and $L_\om(m_n)\in L^{p'}_\nu$ for all $n\in\N$ by
\eqref{eq:ini}. For each $f\in L^p_\nu$, choose a sequence of simple functions $\{s_k\}$ which converges to $f$ in $L^p_\nu$.
Then, since $P_\om:L^p_\nu\to L^p_\nu$ is bounded by the hypothesis, \eqref{eq:i2} yields
	\begin{equation*}
  \begin{split}
  \langle P_\om(f),m_n\rangle_{L^2_\nu}
  =\lim_{k\to\infty}\langle P_\om(s_k),m_n\rangle_{L^2_\nu}
  =\lim_{k\to\infty}\langle s_k,L_\om(m_n)\rangle_{L^2_\nu}
  =\langle f,L_\om(m_n)\rangle_{L^2_\nu},
  \end{split}
	\end{equation*}
and hence $P^\star_\om(m_n)=L_\om(m_n)$ almost everywhere on $\D$.

Since $\om\in\DD$ by the hypothesis of the theorem and $P_\om:L^p_\nu\to L^p_\nu$ is bounded, \cite[Proposition~19(i)]{PelaezRattya2019} implies $\nu_1\in\DD$. Therefore \cite[Lemma~2.1]{PelSum14} yields $\|m_n\|^{p'}_{L^{p'}_\nu}=2\nu_{np'+1}\asymp\nugu\left(1-\frac{1}{n}\right)\asymp\nu_{2n+1}$ and $\om_{2n+1}\asymp\omgu\left(1-\frac{1}{n}\right)$ for all $n\in\N$. Now that $P^*_\om: L^{p'}_\nu\to L^{p'}_\nu$ is bounded, we deduce
	\begin{equation*}
	\begin{split}
	\nugu\left(1-\frac{1}{n}\right)
	&\asymp\|m_n\|^{p'}_{L^{p'}_\nu}\gtrsim \| P^\star_\om(m_n) \|^{p'}_{L^{p'}_\nu}
	=2\left(\frac{\nu_{2n+1}}{\om_{2n+1}} \right)^{p'}\int_{0}^1 s^{np'+1}\left(\frac{\om(s)}{\nu(s)^{\frac1p}}\right)^{p'}\,ds\\
	&\gtrsim \left( \frac{\nugu\left(1-\frac{1}{n}\right)}{\omgu\left(1-\frac{1}{n}\right)}\right)^{p'}
	\sgu\left(1-\frac{1}{n}\right),\quad n\in\N\setminus\{1\}.
	\end{split}
	\end{equation*}
which implies
	$$
	\sup_{1/2\le r<1}\frac{\nugu(r)^{\frac{1}{p}}\sgu(r)^{\frac{1}{p'}}}{\omgu(r)}<\infty
	$$
since $\om_1\in\DD$. Then bearing in mind \eqref{eq:ini},
it follows that $A_p(\om,\nu)<\infty$. 
\end{Prf}

\medskip

Now we proceed with the proof of the convenient reformulation of \eqref{twoweight} in terms of an appropriate duality.

\medskip

\begin{Prf}{\em{Theorem~\ref{th:Pwdospesos}.}}
Assume first that  $(A^{p'}_\sigma)^\star\simeq A^{p}_\nu$ via the $A^2_\om$-pairing with equivalence of norms. Let $H\in L^p_\nu$ and define $h=H\left(\frac{\om}{\nu}\right)^{-p'/p}$ outside of the set on which $\om$ vanishes. Then $\|h\|_{L^{p}_\sigma}=\|H\|_{L^{p}_\nu}$. Consider the bounded linear functional $T_h(f)=\langle f,h\rangle_{L^2_\sigma}$ on $L^{p'}_{\sigma}$ that satisfies $\|T_h\|=\|h\|_{L^{p}_\sigma}$. Then $T_h(f)=\langle f,H\rangle_{L^2_\om}$ and $\|T_h\|=\|h\|_{L^{p}_\sigma}=\|H\|_{L^{p}_\nu}$. For each polynomial $p$, with the Maclaurin coefficients $\{\widehat{p}(n)\}$, Fubini's theorem yields
    \begin{equation}
    \begin{split}\label{eq:du1}
    T_h(p)&=\langle P_\om(p),H\rangle_{L^2_\om}
    =\lim_{r\to1^-}\int_{D(0,r)}P_\om(p)(z)\overline{H(z)}\om(z)\,dA(z)\\
    &=\lim_{r\to1^-}\int_{\D}p(\z)
		\left(\sum_{n=0}^\infty\frac{\int_{D(0,r)}\overline{H(z)}z^n \om(z)\,dA(z)}{2\om_{2n+1}} \overline{\z}^n\right)\om(\z)\,dA(\z)\\
    &=\lim_{r\to 1^-}\sum_{n=0}^\infty\widehat{p}(n)\int_{D(0,r)}\overline{H(z)}z^n\om(z)\,dA(z)
    =\langle p,P_\om(H)\rangle_{L^2_\om}.
    \end{split}
    \end{equation}
By the hypothesis, there exists $g\in A^{{p}}_{\nu}$ such that $T_h(f)=\langle f,g\rangle_{A^2_\om}$ for all $f\in A^{p'}_\sigma$, and $\|T_h\|\asymp\|g\|_{A^{p}_\nu}$. In particular, $T_h(m_n)=\langle m_n,g\rangle_{A^2_\om}=\langle m_n,P_\om(H)\rangle_{A^2_\om}$ for each monomial $m_n(z)=z^n$ with $n\in\N\cup\{0\}$. It follows that $P_\om(H)=g$, and hence $\|P_\om(H)\|_{A^{{p}}_{\nu}}\asymp\|T_h\|=\|H\|_{L^{p}_{\nu}}$.
Since $H\in L^{p}_\nu$ was arbitrarily taken, this shows that $P_\om: L^{p}_\nu\to L^{p}_{\nu}$ is bounded.

Reciprocally, assume   that  $P_\om:L^p_\nu\to L^p_\nu$ is bounded, then  \eqref{eq:apcondition} holds by Lemma~\ref{le:nec},  and in particular $\sigma$ is a weight.
 Let $T\in (A^{p'}_\sigma)^\star$, then $T$
can be extended to a bounded linear functional $\widetilde{T}$ on $L^{p'}_\sigma$ with $\|T\|=\|\widetilde T\|$ by the Hahn-Banach theorem. By arguing as in the first part of the proof, we see that $(L^{p'}_\sigma)^\star$ is isometrically isomorphic to $L^{p}_{\nu}$ via the $L^2_\om$-pairing. Hence there exists $H\in L^{p}_\nu$ such that $\widetilde{T}(f)=\langle f,H\rangle_{L^2_\om}$ for all $f\in L^{p'}_\sigma$, and $\|\widetilde T\|=\|H\|_{L^{p}_\nu}$. Now, for each polynomial $p$, with Maclaurin coefficients $\{\widehat{p}(n)\}$, 
\eqref{eq:du1}
 yields $$ T(p)=\langle P_\om(p),H\rangle_{L^2_\om}
    =\langle p,P_\om(H)\rangle_{L^2_\om}
    =\langle p,g\rangle_{L^2_\om},
   $$
where $g=P_\om(H)\in A^{p}_\nu$ satisfies $\|g\|_{A^{p}_\nu}\le\|P_\om\|_{L^p_\nu\to L^p_\nu}\|H\|_{L^{p}_{\nu}}=\|P_\om\|_{L^p_\nu\to L^p_\nu}\|T\|<\infty$ by the hypothesis.
Bearing in mind that the polynomials are dense in $A^{p'}_\sigma$, because $\sigma$ is radial, and using H\"older's inequality, it follows that  $T(f)=\langle f,g\rangle_{A^2_\om}$ for all $f\in A^{p'}_\sigma$ with $\|g\|_{A^{p}_\nu}\lesssim\|T\|$. Since in addition, each functional $T_g(f)=\langle f,g\rangle_{A^2_\om}$ induced by $g\in A^{p}_\nu$, belongs to $(A^{p'}_\sigma)^\star$ with $\|T_g\|\le\|g\|_{A^{p}_\nu}$ by H\"older's inequality, we deduce $(A^{p'}_\sigma)^\star\simeq A^{p}_\nu$ via the $A^2_\om$-pairing with equivalence of norms.
\end{Prf}

\medskip 

The proof of Theorem~\ref{pr:potencias} requires some preparations. In particular, we will use a suitable norm in the weighted Bergman space induced by a weight in $\DD$. For that purpose some more notation are in order.
The Hadamard product of $f(z)=\sum_{k=0}^\infty \widehat{f}(k)z^k\in\H(\D)$ and $g(z)=\sum_{k=0}^\infty \widehat{g}(k)z^k\in\H(\D)$ is
    $$
    (f\ast g)(z)=\sum_{k=0}^\infty\widehat{f}(k)\widehat{g}(k)z^k,\quad z\in\D.
    $$
 For a Banach space $X\subset\H(\D)$, $s\in\mathbb{R}$, $0<q\le\infty$ and a sequence of polynomials $\{P_n\}_{n=0}^\infty$, let
    \begin{equation}\label{eq:lqXPn}
    \ell^q_s(X,\{P_n\})
    =\left\{f\in\H(\D):\|f\|^q_{\ell^q_s(X,\{P_n\})}=\sum_{n=0}^\infty\left(2^{-ns}\|P_n*f\|_X\right)^q<\infty\right\}, \quad 0<q<\infty,
    \end{equation}
and
    \begin{equation}\label{eq:linftyXPn}
    \ell^\infty_s(X,\{P_n\})
    =\left\{f\in\H(\D):\|f\|_{\ell^\infty_s(X,\{P_n\})}=\sup_n\left(2^{-ns}\|P_n*f\|_X\right)<\infty\right\}.
    \end{equation}
Further, for quasi-normed spaces $X,Y\subset\H(\D)$, let $[X,Y]=\{g\in\H(\D):f*g\in Y\,\text{for all} \,f\in X\}$ denote the space of coefficient multipliers from $X$ to $Y$ equipped with the norm
    $$
    \|g\|_{[X,Y]}=\sup\left\{\|f*g\|_Y:f\in X,\,\|f\|_X\le1\right\}.
    $$
    
With these preparations we are ready for the proof.

\medskip

\begin{Prf}{\em{Theorem~\ref{pr:potencias}}.}
Bearing in mind Lemma~\ref{le:nec} and 
Theorem~\ref{th:Pwdospesos} we only have to prove that
$\left(A^{p'}_\sigma\right)^\star\simeq A^p_\nu$ via the $A^2_\om$-pairing with equivalence of norms,
in order to ensure that $P_\om$ is bounded on $L^p_\nu$.
Observe that $\sigma=\om\widehat{\om_1}^{-\alpha\frac{p'}{p}}$, and hence the condition $A_p(\om,\nu)<\infty$ is equivalent to the inequality $-1<\alpha<\frac{p}{p'}$. Assume, without loss of generality, that $\widehat{\om_1}(0)=1$, and consider a sequence $\{\r_n\}_{n=0}^\infty=\{\r_n(\om)\}_{n=0}^\infty$ such that $\widehat{\om_1}(\r_n)=2^{-n}$ for all $n$. For $n\in\N$, define $M_n=M_n(\om)=E\left(\frac{1}{1-\r_{n}}\right)$, where $E(x)\in\N$ such that $E(x)\le x<E(x)+1$. Write
    $$
    I(0)=I_{\om}(0)=\left\{k\in\N\cup\{0\}:k<M_1\right\}
    $$
and
   \begin{equation*}
    I(n)=I_{\om}(n)=\left\{k\in\N:M_n\le k<M_{n+1}\right\},\quad n\in\N.
  \end{equation*}
For $f\in\H(\D)$, with the Maclaurin series $f(z)=\sum_{n=0}^\infty\widehat{f}(n)z^n$, define the polynomials $\Delta^{\om}_nf$ by
    $$
    \Delta_n^{\om}f(z)=\sum_{k\in I_{\om}(n)}\widehat{f}(k)z^k,\quad n\in\N\cup\{0\}.
    $$
Then, obviously, $f=\sum_{n=0}^\infty\Delta_n^\om f$. Observe that $\nu,\sigma\in \DD$ because $\om\in\DD$ and moreover
	$$
	\widehat{\nu_1}(\r_n)=\frac{2^{-n(1+\a)}}{1+\a}
	\quad\textrm{and}\quad\widehat{\sigma_1}(\r_n)=\frac{2^{-n\left(1-\a\frac{p'}{p}\right)}}{1-\a\frac{p'}{p}},\quad n\in\N\cup\{0\}.
	$$
Therefore,  \cite[Theorem~3.4]{PRAntequera}, see also \cite[Theorem~4.1]{PavMixnormI} and \cite[Theorem~4]{PelRathg} for similar results, shows that
    \begin{equation}\label{18p}
    \|f\|_{A^p_\nu}^p
    \asymp\sum_{n=0}^\infty \widehat{\nu_1}(\r_n)\|\Delta^{\om}_n f\|_{H^p}^p\asymp\sum_{n=0}^\infty 2^{-n(1+\a)}\|\Delta^{\om}_n f\|_{H^p}^p ,\quad f\in\H(\D),
    \end{equation}
    and
    \begin{equation}\label{18pprima}
    \|f\|_{A^{p'}_\sigma}^{p'}
    \asymp\sum_{n=0}^\infty \widehat{\sigma_1}(\r_n)\|\Delta^{\om}_n f\|_{H^{p'}}^{p'}\asymp
    \sum_{n=0}^\infty 2^{-n\left(1-\a\frac{p'}{p}\right)}\|\Delta^{\om}_n f\|_{H^{p'}}^{p'} ,\quad f\in\H(\D).
    \end{equation}
Thus $A^p_\nu=\ell^p_{\frac{\alpha+1}{p}}(H^p,\{\Delta^{\om}_n\})$ and $A^{p'}_\sigma=\ell^{p'}_{\frac{1}{p'}-\frac{\alpha}{p}}(H^{p'},\{\Delta^{\om}_n\})$ with equivalent norms. 

In addition,   $\left(A^{p'}_\sigma\right)^\star$ can be identified with $[A^{p'}_\sigma,H^\infty]$ 
via the $H^2$-pairing with equivalence of norms by \cite[Proposition~1.3]{PavMixnormI}. That is, for each $L\in\left(A^{p'}_\sigma\right)^\star$ there exists a unique $g\in[A^{p'}_\sigma, H^\infty]$ such that $\|L\|=\|g\|_{[A^{p'}_\sigma,H^\infty]}$ and $L(f)=\langle f, g\rangle_{H^2}$ for all $f\in A^{p'}_\sigma$, and conversely, each functional $L(f)=\langle f,g\rangle_{H^2}$, induced by $g\in[A^{p'}_\sigma,H^\infty]$, belongs to $\left(A^{p'}_\sigma\right)^\star$ and satisfies $\|L\|=\| g\|_{[A^{p'}_\sigma,H^\infty]}$. Now 
use 
\cite[Theorem~5.4]{PavMixnormI} and the fact that $(H^{p'})^\star\simeq[H^{p'},H^\infty]=H^{p}$ via the $H^2$-pairing to deduce
    \begin{equation}\label{eq:dualap1}
    \left[\ell^{p'}_{\frac{1}{p'}-\frac{\alpha}{p}}(H^{p'},\{\Delta^{\om}_n\}),H^\infty\right]
    =\ell^{p}_{\frac{\alpha}{p}-\frac{1}{p'}}\left([H^{p'},H^\infty],\{\Delta^{\om}_n\}\right)
    =\ell^{p}_{\frac{\alpha}{p}-\frac{1}{p'}}\left(H^p,\{\Delta^{\om}_n\}\right)
    \end{equation}
with equivalent norms. Since $A^{p'}_\sigma=\ell^{p'}_{\frac{1}{p'}-\frac{\alpha}{p}}(H^{p'},\{\Delta^{\om}_n\})$ with equivalent norms,  it follows that
    \begin{equation}\label{eq:ap1}
    \left(A^{p'}_\sigma\right)^\star
    \simeq[A^{p'}_\sigma,H^\infty]=\ell^{p}_{\frac{\alpha}{p}-\frac{1}{p'}}\left(H^p,\{\Delta^{\om}_n\}\right),
    \end{equation}
via the $H^2$-pairing with equivalence of norms.

Finally, define the operator $I^\om:\H(\D)\to\H(\D)$ by $I^\om(g)(z)=\sum_{k=0}^\infty \widehat{g}(k)\om_{2k+1}z^{k}$ for all $z\in\D$.
The sequence $\{\om_{2k+1}\}_{k=0}^\infty$ is non-increasing,
and hence there exists a constant $C>0$ such that
    $$
    C^{-1}\om_{1+M_{n+1}}\|\Delta^{\om}_n g\|_{H^{p}}
    \le\|\Delta^{\om}_n I^\om(g)\|_{H^{p}}
    \le C\om_{1+M_n}\|\Delta^{\om}_n g\|_{H^{p}},\quad g\in\H(\D),
    $$
by \cite[Lemma~E]{PelRathg}. This combined with $(\om_1)_x\asymp\widehat{\om_1}\left(1-\frac1x\right)$, valid for all $1\le x<\infty$ by \cite[Lemma~2.1(vi)]{PelSum14}, gives $\|\Delta^{\om}_nI^\om(g)\|_{H^{p}}\asymp 2^{-n}\|\Delta^{\om}_n g\|_{H^{p}}$ for all $g\in\H(\D)$.
Therefore \eqref{18p} yields
    \begin{equation*}
    \begin{split}
    \| I^\om(g)\|^{p}_{\ell^{p}_{\frac{\alpha}{p}-\frac{1}{p'}}\left(H^p,\{\Delta^{\om}_n\}\right)}
    &=\sum_{n=0}^\infty  2^{-n\left(\alpha-\frac{p}{p'}\right)}\|\Delta^{\om}_n I^\om(g)\|^{p}_{H^{p}}
    \asymp \sum_{n=0}^\infty  2^{-n\left(\alpha+p-\frac{p}{p'}\right)}\|\Delta^{\om}_n g\|^{p}_{H^{p}}\\
    &=\sum_{n=0}^\infty  2^{-n(\alpha+1)}\|\Delta^{\om}_n g\|^{p}_{H^{p}}
    \asymp\| g\|^{p}_{A^{p}_\nu},\quad g\in\H(\D),
    \end{split}
    \end{equation*}
and hence $I^\om:A^{p}_\nu \to\ell^{p}_{\frac{\alpha}{p}-\frac{1}{p'}}\left([H^p,H^\infty],\{\Delta^{\om}_n\}\right)$ is isomorphic and onto. This converts $\left(A^{p'}_\sigma\right)^\star\simeq\ell^{p}_{\frac{\alpha}{p}-\frac{1}{p'}}\left(H^p,\{\Delta^{\om}_n\}\right)$ via the $H^2$-pairing, valid by \eqref{eq:ap1}, to $\left(A^{p'}_\sigma\right)^\star\simeq A^{p}_\nu$ via the $A^2_\om$-pairing, and thus finishes the proof of the first statement.

Now assume that $A_p(\om,\nu)<\infty$ and $\om\in\DDD$. Then a calculation shows that $\nu=\om\widehat{\om}^\alpha \in\DDD$, 
and $P^+_\om: L^p_\nu\to L^p_\nu$ is bounded by \cite[Theorem~13]{PelaezRattya2019}.
Reciprocally, if  $P^+_\om: L^p_\nu\to L^p_\nu$ is bounded, then $A_p(\om,\nu)<\infty$ and $\om\in\DDD$ by \cite[Theorem~13]{PelaezRattya2019}.
Alternatively, we can use Lemma~\ref{le:nec} and \cite[Proposition~20]{PelaezRattya2019} to get the conditions
$A_p(\om,\nu)<\infty$ and $\om\in\DDD$.

\end{Prf}

\section{Calder\'on and Stieltjes operators, and H\"ormander maximal function}\label{s:real}
 
We will prove the following reformulation of Theorem~\ref{th:calderon} using ideas coming from~\cite{DMRO}. 
 
\begin{theorem}\label{th:calderonref}
Let $1<p<\infty$ and let $\om$ be a radial weight. Let $\nu:[0,1)\to \mathbb{R}^+$ such that $\nu\in L^1_\om([0,1))$. Then the following statements are equivalent:
	\begin{enumerate}
	\item[\rm(i)] $H_\om+H_\om^\star :L^p_{\om\nu}\to L^p_{\om\nu}$ is bounded;
	\item[\rm(ii)] $S_{\om_1,\rad} :L^p_{\om\nu}\to L^p_{\om\nu}$ is bounded;
	\item[\rm(iii)] $M_{\om_1,\rad} :L^p_{\om\nu}\to L^p_{\om\nu}$ is bounded;
	\item[\rm(iv)] $ A_p(\om,\om\nu)<\infty$.	
	\end{enumerate}
Moreover,
	\begin{equation}\label{giguli}
	\|H_\om+H_\om^\star\|_{L^p_{\om\nu}\to L^p_{\om\nu}}
	\asymp\|S_{\om_1,\rad}\|_{L^p_{\om\nu}\to L^p_{\om\nu}}
	\asymp\|M_{\om_1,\rad}\|_{L^p_{\om\nu}\to L^p_{\om\nu}}
	\asymp A_p(\om,\om\nu).
	\end{equation}
\end{theorem}   

Before embarking on the proof of Theorem~\ref{th:calderonref}, two auxiliary results are in order. For a radial weight $\eta$, denote 
	\begin{equation}\label{eq:M}
	M_\eta f(t)=\sup_{b<t}\frac{\int_b^1 |f(s)|\eta(s)\,ds}{\widehat{\eta}(b)},\quad0<t<1.
	\end{equation} 
By denoting $f_{\theta}(r)=f(re^{i\theta})$ for all $re^{i\theta}\in\D$, \eqref{eq:maximal} shows that $M_{\eta,\rad}(f)(z)= M_{\eta}f_{\theta}(r)$. Further, for a given measurable function $f$, $k\in\mathbb{Z}$ and a radial weight $\om$, define
	$$
	E_k=E_k(f_\theta,\om)=\{2^{k}< M_{\om_1}f_{\theta}(r)\le 2^{k+1}\},\quad 0\le \theta<2\pi.
	$$
For each $\t$, the function $M_{\om_1}f_{\theta}$ is nondecreasing, and hence $E_k=(b_k(\t), b_{k+1}(\t)]$ for some  $b_k(\t)< b_{k+1}(\t)$ and 
	\begin{equation}\label{bk}
	\frac{\int_{b_k(\t)}^1 f_{\t}(s)\om_1(s)\,ds}{\omgu(b_k(\t))}=2^k,\quad k\in\Z,\quad 0\le \theta<2\pi.
	\end{equation}
It follows from \eqref{bk} that
	\begin{equation}\label{bk1}
	\frac{\omgu(b_{k+1}(\t))}{\omgu(b_k(\t))}\le\frac12,\quad k\in\Z,\quad 0\le \theta<2\pi,
	\end{equation} 
and therefore
	$$
	\omgu(b_k(\t))
	=\int_{b_k(\t)}^{b_{k+1}(\t)} \om_1(s)\,ds+ \omgu(b_{k+1}(\t))
	\le \int_{b_k(\t)}^{b_{k+1}(\t)} \om_1(s)\,ds + \frac12\omgu(b_{k}(\t)).
	$$
Thus
	\begin{equation}\label{bk2}
	\omgu(b_{k}(\t))
	\le 2\int_{b_k(\t)}^{b_{k+1}(\t)}
	\om_1(s)\,ds,\quad k\in\Z,\quad 0\le \theta<2\pi.
	\end{equation} 

\begin{lemma}\label{le:apbk}
Let $1<p<\infty$ and let $\om$ be a radial weight. Let $\nu:[0,1)\to \mathbb{R}^+$ such that $A_p(\om,\om\nu)<\infty$. Then
there exists a constant $C=C(\om,\nu,p)>0$ such that
	$$
	\int_{b_k(\t)}^{1}\nu(s)^{-\frac{p'}{p}}\om(s)s\,ds
	\le C\int_{b_k(\t)}^{b_{k+1}(\t)}\nu(s)^{-\frac{p'}{p}}\om(s)s\,ds,\quad 0\le \theta<2\pi.
	$$ 
\end{lemma}

\begin{proof}
The definition of $A_p(\om,\om\nu)$, \eqref{bk2} and H\"older's inequality yield
	\begin{equation*}
	\begin{split}
	&\left(\int_{b_k(\t)}^1 \nu(s)\om(s)s\,ds \right)^{\frac{1}{p}}
	\left(\int_{b_k(\t)}^1 \nu(s)^{-\frac{p'}{p}}\om(s)s\,ds \right)^{\frac{1}{p'}}
	\\ & \qquad\lesssim\omgu(b_k(\t))
	\lesssim\int_{b_k(\t)}^{b_{k+1}(\t)}\om_1(s)\,ds
	\\
	&\qquad\le\left(\int_{b_k(\t)}^{b_{k+1}(\t)}\nu(s)\om(s)s\,ds \right)^{\frac{1}{p}}
	\left(\int_{b_k(\t)}^{b_{k+1}(\t)} \nu(s)^{-\frac{p'}{p}}\om(s)s\,ds \right)^{\frac{1}{p'}}\\
	&\qquad\le\left(\int_{b_k(\t)}^{1}\nu(s)\om(s)s\,ds \right)^{\frac{1}{p}}
	\left(\int_{b_k(\t)}^{b_{k+1}(\t)} \nu(s)^{-\frac{p'}{p}}\om(s)s\,ds \right)^{\frac{1}{p'}},
	\end{split}
	\end{equation*}
and the assertion follows.
\end{proof}

\begin{lemma}\label{le:Meta}
Let $1<p<\infty$ and let $\eta$ be radial weight. Then $M_\eta: L^1([0,1),\eta)\to L^{1,\infty}([0,1),\eta)$ is bounded with norm at most one. In particular, for each $1<p<\infty$, $M_\eta:L^p([0,1),\eta)\to L^{p}([0,1),\eta)$ is bounded
with norm depending only on $p$.
\end{lemma}
 
\begin{proof}
Since $M_\eta f$ is an increasing function, for each $\lambda>0$ and $f\in L^1([0,1),\eta)$, the level set $\{ t\in(0,1):\,M_\eta f(t)>\lambda\}$ coincides with the interval $(b,1)$ for some $b=b(f,\eta,\lambda)\in(0,1)$ such that $\int_b^1|f(s)|\eta(s)\,ds=\widehat{\eta}(b)\lambda$. Therefore 
	$$
	\int_{\{s\in(0,1):\,M_\eta f(s)>\lambda\}}\eta(s)\,ds
	=\widehat{\eta}(b)
	=\frac{\int_b^1|f(s)|\eta(s)\,ds}{\lambda}
	\le\frac{\int_0^1|f(t)|\eta(s)\,ds}{\lambda},
	$$
and thus $M_\eta:L^1([0,1),\eta)\to L^{1,\infty}([0,1),\eta)$ is bounded with norm at most one. 

Now that $M_\eta$ is clearly bounded from $L^\infty([0,1),\eta)$ into itself with norm at most one, the statement for $1<p<\infty$ follows from the Marcinkiewicz interpolation theorem. 
\end{proof}

\begin{Prf}{\em{Theorem~\ref{th:calderonref}.}}
We begin with showing that (iii) and (iv) are equivalent. Assume first (iv), that is, $A_p(\om,\om\nu)<\infty$. Then, for each $0\le\t<2\pi$,
by the definition of $E_k$, \eqref{bk} and the hypothesis (iv), we have
	\begin{equation*}
	\begin{split}
	I(f_\theta)
	&=\int_{0}^1\left( M_{\om_1}f_{\t}(r)\right)^p \nu(r)\om_1(r)\,dr
	=\sum_{k\in\mathbb{Z}}\int_{E_k} \left( M_{\om_1}f_{\t}(r)\right)^p\nu(r)\om_1(r)\,dr\\
	&\le2^p\sum_{k\in\mathbb{Z}}2^{kp}\int_{b_k(\theta)}^{b_{k+1}(\theta)}\nu(r)\om_1(r)\,dr\\
	&=2^p\sum_{k\in\mathbb{Z}} \frac{\left(\int_{b_k(\t)}^1 f_{\t}(r)\om_1(r)\,dr\right)^p
	\int_{b_k(\theta)}^{b_{k+1}(\theta)}\nu(r)\om_1(r)\,dr}{\omgu(b_k(\theta))^p}\\
	&\le 2^p\sum_{k\in\mathbb{Z}}\frac{\left(\int_{b_k(\t)}^1 f_{\t}(r)\om_1(r)\,dr\right)^p 
	\int_{b_k(\theta)}^{1}\nu(r)\om_1(r)\,dr}{\omgu(b_k(\theta))^p}\\
	&\le2^pA^p_p(\om,\om\nu)\sum_{k\in\mathbb{Z}}\left(\frac{\int_{b_k(\t)}^1 f_{\t}(r)\om_1(r)\,dr}
	{\int_{b_k(\t)}^{1}\nu^{-\frac{p'}{p}}(r)\om_1(r)\,dr}\right)^p
	\int_{b_k(\t)}^{1}\nu^{-\frac{p'}{p}}(r)\om_1(r)\,dr.
	\end{split}
	\end{equation*}
By applying now Lemmas~\ref{le:apbk} and~\ref{le:Meta}, we deduce
	\begin{equation*}
	\begin{split}
	I(f_\theta)
	&\lesssim A^p_p(\om,\om\nu)\sum_{k\in\mathbb{Z}}
	\left(\frac{\int_{b_k(\t)}^1 f_{\t}(r)\om_1(r)\,dr}{\int_{b_k(\t)}^{1}\nu^{-\frac{p'}{p}}(r)\om_1(r)\,dr}\right)^p
	\int_{b_k(\t)}^{b_{k+1}(\t)}\nu^{-\frac{p'}{p}}(r)\om_1(r)\,dr\\
	&\le A^p_p(\om,\om\nu)\int_{0}^{1}\left(M_{\nu^{-\frac{p'}{p}}\om_1}(f_{\t}\nu^{\frac{p'}{p}})(r)\right)^p\nu^{-\frac{p'}{p}}(r)\om_1(r)\,dr\\
	&\lesssim A^p_p(\om,\om\nu)\int_{0}^{1}\left(|f_{\t}(r)|\nu^{\frac{p'}{p}}(r)\right)^p\nu^{-\frac{p'}{p}}(r)\om_1(r) \,dr\\
	&=A^p_p(\om,\om\nu)\int_{0}^{1}|f_{\t}(r)|^p\nu(r)\om_1(r)\,dr,\quad 0\le \t<2\pi.
	\end{split}
	\end{equation*}
Since
	\begin{equation*}
	\begin{split}
	\int_{\D}\left(M_{\om_1,\rad}(f)(z)\right)^p\nu(z)\om(z)\,dA(z)
	=\frac1{\pi}\int_{0}^{2\pi}I(f_\theta)\,d\theta,
	\end{split}
	\end{equation*}
this shows that $M_{\om_1,\rad} :L^p_{\om\nu}\to L^p_{\om\nu}$ is bounded with
	\begin{equation}\label{pili}
	\|M_{\om_1,\rad}\|_{L^p_{\om\nu}\to L^p_{\om\nu}}\lesssim A_p(\om,\om\nu).
	\end{equation}

Conversely, assume that (iii) holds, that is, $M_{\om_1,\rad} :L^p_{\om\nu}\to L^p_{\om\nu}$ is bounded. Then it follows that	
	$$
	\int_{0}^1\left(M_{\om_1}f(r)\right)^p\nu(r)\om_1(r)\,dr
	\le\|M_{\om_1,\rad}\|_{L^p_{\om\nu}\to L^p_{\om\nu}}^p\int_0^1f(r)^p\om_1(r)\nu(r)\,dr
	$$
for all non-negative radial functions $f$. By choosing $f(r)=\nu^{-\frac{p'}{p}}(r)\chi_{[a,1)}(r)$, where $a\in [0,1)$, and using standard arguments, we obtain 
	\begin{equation}\label{pili2}
	A_p(\om,\om\nu)\le\|M_{\om_1,\rad}\|_{L^p_{\om\nu}\to L^p_{\om\nu}}.
	\end{equation}

The pointwise inequality 
	\begin{equation*}
	\begin{split}
	S_{\om_1,\rad}(f)(z)
	&\le\frac{ \int_{|z|}^1f\left(s\frac{z}{|z|}\right) s\om(s)\,ds}{\omgu(z)}+\int_{0}^{|z|} f\left(s\frac{z}{|z|}\right)\frac{\om_1(s)}{\omgu(s)}\,ds\\
	&=H_\om(f)(z)+H^\star_\om(f)(z),\quad z\in\D,\quad f\ge 0,
	\end{split}
	\end{equation*} 
yields 
	\begin{equation}\label{eq:CdominaS}
	\begin{split}
	\|S_{\om_1,\rad}\|_{L^p_{\om\nu}\to L^p_{\om\nu}}\le\|H_\om+H^\star_\om\|_{L^p_{\om\nu}\to L^p_{\om\nu}},
	\end{split}
	\end{equation}
and thus (i) implies (ii).

If $b<|z|$, then we have 
	$$ 
	\frac{\int_{b}^1f\left(s\frac{z}{|z|}\right)\om(s)s\,ds}{\omgu(b)}
	\le2\int_{b}^1 \frac{f\left(s\frac{z}{|z|}\right)}{\omgu(b)+\omgu(z)}\om(s)s\,ds
	\le2S_{\om_1,\rad}(f)(z),\quad z\in\D,\quad f\ge 0,
	$$
which implies $M_{\om_1,\rad}(f)\le 2S_{\om_1,\rad}(f)$ for all $f\ge 0$. Therefore
	\begin{equation}\label{eq:CdominaS2}
	\|M_{\om_1,\rad}\|_{L^p_{\om\nu}\to L^p_{\om\nu}}\le\|S_{\om_1,\rad}\|_{L^p_{\om\nu}\to L^p_{\om\nu}},
	\end{equation}
and thus (ii) implies (iii). 

To complete the proof, it remains to show that (iv) implies (i) and 
$$\|H_\om+H_\om^\star\|_{L^p_{\om\nu}\to L^p_{\om\nu}}\lesssim A_p(\om,\om\nu).$$ Assume (iv), that is, $A_p(\om,\om\nu)<\infty$. Then $\om\nu^{-\frac{p'}{p}}$ is a weight, and therefore $(L^p_{\om\nu})^\star\simeq L^{p'}_{\om\nu^{-\frac{p'}{p}}}$ and $(L^{p'}_{\om\nu^{-\frac{p'}{p}}})^\star\simeq L^{p}_{\om\nu}$ via the $L^2_\om$-pairing. Therefore each linear operator $T$ is bounded from $L^{p}_{\om\nu}$ into itself if and only if its adjoint $T^\star$, defined by the identity $\langle Tf,g\rangle_{L^2_\om}=\langle f, T^\star g\rangle_{L^2_\om}$, is bounded from $L^{p'}_{\om\nu^{-\frac{p'}{p}}}$ into itself, and the operator norms are comparable. Since $A_p(\om,\om\nu)<\infty$ by the assumption, \eqref{pili} yields $\|H_{\om_1,\rad}\|_{L^p_{\om\nu}\to L^p_{\om\nu}}\lesssim A_p(\om,\om\nu)<\infty$. Moreover, $A_p(\om,\om\nu)=A_{p'}(\om,\om\nu^{-\frac{p'}{p}})$, and hence \eqref{pili} shows that $H_{\om_1,\rad} :L^{p'}_{\om\nu^{-\frac{p'}{p}}}\to L^p_{\om\nu^{-\frac{p'}{p}}}$ is bounded and the operator norm is bounded by a constant times $A_p(\om,\om\nu)$. It follows that $H^\star_{\om_1,\rad}:L^p_{\om\nu}\to L^p_{\om\nu}$ is bounded, and $\|H^\star_{\om_1,\rad}\|_{L^p_{\om\nu}\to L^p_{\om\nu}}\lesssim A_p(\om,\om\nu)$. Therefore 
	\begin{equation}\label{kusi}
	\|H_\om+H_\om^\star\|_{L^p_{\om\nu}\to L^p_{\om\nu}}\lesssim A_p(\om,\om\nu), 
	\end{equation}
and thus (i) is satisfied. The estimates \eqref{giguli} for the operator norms follows from the inequalities \eqref{pili}--\eqref{kusi}.
\end{Prf}

\end{document}